\documentclass[twoside,12pt]{amsart}

\usepackage{amssymb,amsfonts,amsmath}
\usepackage{latexsym}
\usepackage[latin1]{inputenc}
\usepackage{graphicx,epsf}
\usepackage[hyperindex, colorlinks,]{hyperref}
\newtheorem{theorem}{Theorem}
\newtheorem{lemma}[theorem]{Lemma}
\newtheorem{proposition}[theorem]{Proposition}
\newtheorem{corollary}[theorem]{Corollary}

\newtheorem{definition}[theorem]{Definition}

\newtheorem{remark}[theorem]{Remark}

\numberwithin{theorem}{section}

\numberwithin{equation}{section}

\def\supp{\hbox{supp}}

\def\eps{\varepsilon}

\def\part{\partial}
\newcommand{\R}{{\mathord{\mathbb R}}}

\begin{document}

\title[Null quadrature domains]
{Null quadrature domains and a free boundary problem for the Laplacian}

\author[L. Karp]{Lavi Karp*}
\email{karp@braude.ac.il}
\address{Department of Mathematics,
         ORT Braude College,
         P.O. Box 78,
         21982 Karmiel,
         Israel}\thanks{*Supported by ORT
Braude College's Research Authority}

\author[A.S. Margulis]{Avmir S. Margulis}
\email{avmirs@yahoo.com}
\address{6 Hakozrim St., Herzliya, Israel}



\dedicatory{Dedicated  to Harold S. Shapiro}

\begin{abstract} Null quadrature domains are unbounded domains in
$\R^n$ ($n \geq 2$) with external gravitational force zero in some
generalized sense. In this paper we prove that the complement of null
quadrature domain is a convex set with real analytic boundary.
We establish the quadratic growth estimate for the Schwarz potential of a null
quadrature domain which reduces our main result to Theorem II of \cite{CKS}
on the regularity of solution to the classical global free boundary problem
for Laplacian. We also show that any null quadrature domain with non-zero
upper Lebesgue density at infinity is half-space. 
\end{abstract}

\maketitle


\section{Introduction}

An open set $\Omega$ in $\R^n$ is called a {\it null quadrature
domain} if 
\begin{equation*}
 \int\limits_\Omega hdx=0
\end{equation*}
 for all harmonic and integrable functions $h$ in $\Omega$. This class
of domains includes half-spaces, exterior of ellipsoids, exterior of
strips,  exterior of elliptic paraboloids and cylinders over domains
of these types. In 1980 Sakai showed that in $\R^2$ any   null
quadrature domain belongs to one of the categories above \cite{Sa1}. A
complete description of all null quadrature domains in  higher dimensions
has remained an open problem; the best result was obtained in
\cite{KM}, where the case of  a complement contained in a cylinder of
co-dimension two  was settled.

An equivalent definition of null quadrature domain $\Omega$ is by
means of a free boundary problem and potentials.  If there exists
a solution $u$ to the system
\begin{equation}
\label{eq:intro:1} \Delta u=-\chi_\Omega,\quad u=|\nabla u|=0 \
\text{on}\ \R^n\setminus \Omega,
\end{equation}
then $u$ is called the {\it Schwarz potential} of $\Omega$. The Schwarz
potentials of null quadrature domains are  a subclass of all solutions to
(\ref{eq:intro:1})
(see Section \ref{sec:nfbp}). This subclass is characterized by  an additional
condition, namely, the Schwarz potential $u$ is a (Newtonian) potential of
$\Omega$, and it can be expressed in term of a growth property 
at infinity. Thus a domain $\Omega$ is a null quadrature domain if and only if
there exists a solution $u$ to the free boundary 
problem (\ref{eq:intro:1})  satisfying 
\begin{equation} \label{eq:intro:2}
u(x)=o(|x|^3)\quad\text{as}\ |x|\to\infty, 
\end{equation} 
see Corollary \ref{TM15} below.

Shapiro conjectured that  the complement of a null quadrature domain is a convex
set  \cite{Shp2}, and he showed  it under   some additional restrictions both on
the domain $\Omega$  and its Schwarz potential $u$.

Under the assumption that the Schwarz potential $u$ of $\Omega$ has  a quadratic
growth, Caffarelli, Karp and Shahgholian   proved that the complement,
$\Omega^c$, is convex and the boundary,  $\partial\Omega$, is a real analytic
surface, \cite[Theorem II]{CKS}.

Our principal result confirms Shapiro's conjecture and shows that any null
quadrature domain $\Omega$  has a convex complement and  real analytic
boundary. The main technical tool  is the establishment of the quadratic
growth of solutions to the free boundary problem (\ref{eq:intro:1}) when the a
priori growth condition (\ref{eq:intro:2}) is imposed.

The system (\ref{eq:intro:1}) arises also
in the investigation of the regularity of free boundaries in the
obstacle problem as a ``blow up limits' (see e.g. \cite{Ca2,Fr}).
However, in that case there is an additional condition, namely $u\leq
0$ in $\R^n$. Under this condition, Caffarelli proved that the
complement of $\Omega$  is a convex set \cite{Ca2}. The essence of Theorem II
in \cite{CKS} is showing that any Schwarz potential with a quadratic
growth is non-positive, hence the unsigned  free
boundary problem inherits the same regularity properties as the obstacle
problem in that case.  However, when the growth condition (\ref{eq:intro:2}) is
removed, then there are solutions to (\ref{eq:intro:2}) with   a non-convex
complement, see \cite{Shp2} and  Section \ref{sec:nfbp} for further examples.

The quadratic growth  problem occurs also
in the local version of (\ref{eq:intro:1}),  that is, if
$x_0$ is a point on the free boundary, then a priori $|u(x)|\leq C
|x-x_0|^2\log |x-x_0|$.  In that case the quadratic growth was resolved
in \cite[Theorem I]{CKS}. In order to obtain it in
the global case, we  combine the techniques of \cite{CKS} with
Shapiro's quasi-balayage \cite{Shp5}, and
some estimates of the generalized Newtonian potential \cite{KM}.

The generalized Newtonian potential  is a multi-valued right inverse to the
Laplacian and unique up to a harmonic polynomial of degree not exceeding two.
It enables the 
computation of a potential of arbitrary domains in $\R^n$ in a similar
manner to the ordinary Newtonian potential (see  \cite{Ka1, KM, Ma1}). For
example, the
distribution $u$ which satisfies (\ref{eq:intro:1}) and
(\ref{eq:intro:2}) is a generalized Newtonian potential of the null
quadrature domain $\Omega$. Since $u$ is zero on $\Omega^c$, we see
that  null quadrature domains  consist of open sets in $\R^n$ having
zero gravitational force in their complement. 

Another characterization of null quadrature domains is through the
potential of their complement. Namely, $\Omega$ is a null quadrature
domain if and only if the generalized Newton potential of
$D:=\R^n\setminus\Omega$ coincides with a quadratic polynomial in $D$
(see Theorem \ref{thm:8} below). In case $D$ is bounded, this means
that in a suitable coordinates system the shell $\lambda D\setminus D$
produces no gravitational force in the cavity $D$. The classification
of such domain was settled by Dive \cite{Di}, and Nikliborc \cite{Ni}.
We refer to \cite{MK} for further discussion on that property (see
also \cite{FS, Ka2, KM}).

Thus the main difficulty of the classification of null quadrature
domains occurs when the complement $D:=\R^n\setminus\Omega$ is
unbounded. In this paper we show that if $D$ has a positive  upper
Lebesgue density of at infinity, then it must be a half-space. The
classification when the Lebesgue density of $D$  is zero will be
discussed in \cite{MK}. 

The growth condition (\ref{eq:intro:2})  is indispensable for $\Omega$ being a
null quadrature domain.  Indeed, Gustafsson and Shapiro \cite{Shp2}, and
independently Richardson \cite{Ri}, used conformal mappings in order to
construct  a two dimensional domain $\Omega$ and its  the Schwarz potential $u$,
such  that
$\R^2\setminus\Omega$ is bounded and $u$ has polynomial growth, or even  an
essential singularity,  at infinity. Hence  the Scwharz potential $u$ is  not a
(Newtonian) potential and consequently  the complement  $\R^2\setminus\Omega$ is
not necessarily convex. 
 Shapiro
raised the question whether an analogous example exists in higher
dimensions \cite{Shp2}. We present an $n$-dimensional version of the
above example,  our construction is based upon a theorem on the solvability ``in
the small'' of
the inverse problem of potential theory (see \cite{Is} for details).

The plan of the paper is as follows. In the next Section we provide
some background on generalized Newtonian potential, quadrature domains
and the Schwarz potentials.  Section \ref{sec:nfbp} 
deals with harmonic continuation of potentials, its relations to null quadrature
domain and the construction of a solution to (\ref{eq:intro:1}) with
arbitrarily growth at infinity. In the last section we prove the
quadratic growth of the Schwarz potential and discuss geometric
properties of null quadrature domains.

\section{Background and notions}
\label{sec:background}

In this section we present a brief account of the generalized
Newtonian
potential, quadrature domains and Schwarz potentials.  Throughout of
this paper points in $\R^n$ are denoted by $x$ or $y$, $(x\cdot y)$
is the scalar product and
$|x|=\sqrt{(x\cdot x)}.$ For a set $A \subset
\R^n$,
we denote by $A^c$ the complement of $A$,  $\overline{A}$  the
closure of $A$, $\text{int}(A)$  its interior,  $\part A$ the
boundary of $A$, $\chi_A$ the characteristic function of $A$  and by
$|A|$ is the Lebesgue measure of a
measurable set $A$. 
An open ball with radius $r$ and center $x_0$ is denoted by
$B_r(x_0)$.

\subsection{Generalized Newtonian Potential}

The   definition   and establishment of basic  properties of this
potential were carried out in \cite{KM} (see also \cite{Ka1,Ma1}).  
Here we recall the definition and present a few applications.
A different extension of the Newtonian potential for measures with
non-compact support is given in \cite[Ch.II, \S3]{DL}. But their
extension does not
include the class of densities in $L^\infty(\R^n)$.

The Newtonian potential of a measure $\mu$ with compact support is
defined by means of the convolution
\begin{equation}
\label{eq:27}
    V(\mu)(x)=\left( J\ast\mu \right)(x)=\int J(x-y)d\mu(y),
\end{equation} 
where
\[
    J(x)=\left\{\begin{array}
        {ll}-\dfrac{1}{2\pi}\log |x| ~~, \quad &n=2\\
        \dfrac{1}{(n-2)\omega_n |x|^{n-2}} ~~, \quad &n\geq 3
    \end{array}\right.
\]
and $\omega_n$ is the area of the unit sphere in $\R^n$. The potential
$V(\mu)$ satisfies the Poisson equation $\Delta V(\mu)=-\mu$ in the
distributional sense. The {\it generalized Newtonian
potential} is a multi-valued right inverse of the Laplacian on the
space  $\mathcal{L}$, the space of all Radon measures $\mu$ in $\R^n$
satisfying
condition
\begin{equation} 
\label{eq:1} 
\left\|\mu\right\|_{\mathcal{L}} :=
\int \frac{d|\mu|(x)}{1+|x|^{n+1}} < \infty. 
\end{equation} 
The linear space $\mathcal{L}$ is the Banach space with the norm
defined by (\ref{eq:1}).
For $\mu \in \mathcal{L}$ we define the third order generalized derivatives of the potential
as tempered distributions by setting
\[
    \langle V^\alpha(\mu), \varphi \rangle := - \int \part^\alpha V(\varphi)(x) d\mu(x) ~~,
    \quad \varphi\in \mathcal{S},~~|\alpha|=3,
\]
where $\mathcal{S}$ is the Schwartz class of rapidly decreasing functions. 

\begin{definition}      
\label{Generalized Newtonian potential}
The generalized Newtonian potential $V[\mu]$ of a measure $\mu\in \mathcal{L}$ is the set of
all solutions to the system
\begin{equation}
\label{eq:28}
    \left\{
    \begin{array}
        {l}\Delta u=-\mu \\
        \part^\alpha u=V^\alpha(\mu) ~,\quad |\alpha|=3~~.
    \end{array} \right.
\end{equation}
\end{definition}
The existence of solutions to system (\ref{eq:28}) was proved in
\cite{KM}. The solution of system (\ref{eq:28}) is unique modulo
$\mathcal{H}_2
$, the space of all harmonic polynomials of degree at most two. The
operator $V \colon \mathcal{L} \to \mathcal{S}^\prime / \mathcal{H}_2$
is continuous \cite{KM}.

Here and in the sequel we denote by $V[\sigma]$ the generalized
potential of a measure $\mu = \sigma \lambda$,
where $\lambda$ is the Lebesgue measure and $\sigma
\in L^\infty(\R^n)$. We use the notation $V[A] = V[\chi_A]$ when
 $\sigma= \chi_A$, the characteristic function of a measurable set
$A$. For $v \in V[\mu]$, we may
use the terminology of a generalized Newtonian potential
of $\mu$, or  simply  a potential of $\mu$.

The generalized Newtonian potential is an adequate notion for
explicit calculations of potentials
of unbounded domains and of measures with a non-compact support. The
following two examples  demonstrate it.
By Definition \ref{Generalized Newtonian potential},
\[
    V[\R^n] = - \frac{|x|^2} {2n} + \mathcal{H}_2,
\]
and hence for any measurable set $A$ we have the {\it complementary formula}
\begin{equation}
\label{eq27}
    V[A] + V[A^c] = - \frac{|x|^2} {2n} + \mathcal{H}_2.
\end{equation}

In \cite{KM} we derived an explicit formula for the potential of an
arbitrary cone. Let $K$ be a cone with the vertex at the origin. Then
 $u$ is a generalized Newtonian potential of the cone $K$ if
and only if
\begin{equation}
\label{eq:30}
    u(x)=P_2(x)\log |x| +\Phi(x) +q(x),
\end{equation} 
where $P_2$ is a homogeneous harmonic polynomial of degree two, $\Phi$
is a homogeneous function of order two and $q$ is a linear polynomial.
If $P_2\not\equiv 0$, then the logarithmic term  is
the dominant term both at the origin and  infinity. If $K$ is
convex and different from a half-space, or if $\overline{K}$ is
contained in a "critical cone"
\[
    K_{cr} := \{x : |x|^2 < n (x \cdot a)^2\},
\]
where $a$ is a unit vector, then $P_2 \not \equiv 0$, i.e. the logarithmic term is not zero.
On the other hand, for any $\eps > 0$ there are cones 
    $K \subset \{x : |x|^2 < (n+\eps) (x \cdot a)^2\}$
with $P_2 \equiv 0 $.

Another useful feature is an integral representation of the
generalized potential \cite{Ka1, KM}. Let 
\begin{equation}
\label{eq29} 
J_2(x,y):=  J(x-y)- \chi_{\{y:|y|>1\}} \sum_{|\alpha|\leq
2}\dfrac{(-x)^{\alpha}}{\alpha!}\part^\alpha J(y), 
\end{equation} 
and set 
\begin{equation}
 \label{eq:31} 
V_2(\mu)(x) := \int \limits_{\R^n} J_2(x,y) d\mu(y).
\end{equation}
Then
the potential $V_2(\mu)$ belongs to $V[\mu]$ (cf. \cite{KM}), and for
any $\sigma\in L^\infty(\R^n)$ the following estimate holds (cf.
\cite{Ka1, Shp4}): \begin{equation} \label{eq26} |V_2(\sigma)(x)|\leq
C\|\sigma\|_\infty (1+|x|)^2\log(2+|x|). 
\end{equation} 
Note that any $v \in V[\sigma]$ must satisfy (\ref{eq26}), since
$V[\sigma] = V_2(\sigma)+\mathcal{H}_2$. Therefore, by Liouville's
theorem we get
an equivalent definition of the generalized Newtonian potential for
$\sigma\in L^\infty(\R^n)$:

\begin{proposition} 
\label{thm:4}
A solution $u$ to the Poisson equation $\Delta u=-\sigma$ with
$\sigma\in L^\infty(\R^n)$
is a generalized potential of $\sigma$ if and only if it satisfies the estimate (\ref {eq26}).
\end{proposition}

\subsection{Quadrature Domains  and Schwarz potentials} $~$
 
The mean value theorem for harmonic functions can be written as a 
{\it quadrature identity}
\[
    \int \limits_\Omega hdx = |\Omega| h(x_0) = \int h d\mu, 
\]  
where $\Omega = B_r(x_0)$ is a ball centered at $x_0$ and $\mu$ is
the Dirac measure times
the volume of the ball. This property is equivalent to Newton's
theorem asserting that the ball gravitationally attracts points
outside it as if all its mass were concentrated at the center of the
ball.

A natural generalization of the above property of balls is the notion
of {\it quadrature
domain} (see \cite{GS, Sa2, Shp1}):

\begin{definition} 
Let $\mu$ be a Radon measure in $\R^n$. An open set $\Omega$ is called
a quadrature domain of $\mu$  and with respect to harmonic functions
if
    $\mu \big |_{\R^n \setminus \Omega} = 0$,
and
\begin{equation}
\label{eq:29}
    \int\limits_\Omega hdx =\int h d\mu \quad  \quad\text{for all} \ \ h\in HL^1(\Omega)
\end{equation}
where $HL^1(\Omega)$ is the class of all functions harmonic in $\Omega$ and integrable
over $\Omega$.
\end{definition}

The identity of the external potential
of $\Omega$ and $\mu$ out of $\Omega $ is probably one of the 
 most important property of a quadrature domains.

Observe, if $\Omega$ is a quadrature domain of $\mu$ and
$\overline{\Omega} \ne \R^n$, then we may assume that 
$B_1(0) \subset \Omega^c$. Hence for  any $y \in \Omega$, the kernel $J_2$
defined by (\ref{eq29}) becomes
\[
    J_2(x,y) =  J(x-y)- \sum_{|\alpha|\leq 2}
    \dfrac{(-x)^{\alpha}}{\alpha!}\part^\alpha J(y).
\]
Therefore for any fixed $x \in \Omega^c$ the function $J_2(x,
\cdot)$
and its first order  derivatives belong to $HL^1(\Omega)$. Applying the identity
(\ref{eq:29})
to these functions, we conclude that
\[
    V_2(\Omega)\big|_{\Omega^c} =
    V_2(\mu)\big|_{\Omega^c} \quad \text{and} \quad \nabla
    V_2(\Omega)\big|_{\Omega^c} = \nabla V_2(\mu)\big|_{\Omega^c}.
\]
Hence, $u:= V_2(\Omega-\mu)$ is the solution of the following free
boundary problem: 
\begin{equation}
\label{eq:33}
\left\{\begin{array}{ll}
    \Delta u =-(\chi_\Omega-\mu) \quad &\text{in} \quad \R^n\\
     u=|\nabla u|=0\ \quad &\text{in} \quad  \R^n\setminus\Omega     
 
       \end{array}\right..
\end{equation}
Obviously, if $\overline \Omega \ne \R^n$, then the solution of (\ref{eq:33}) is
unique. This follows from the fact that the difference of any two
solutions is a harmonic
function in $\R^n$ which vanishes  on $\Omega^c$.

\begin{definition} A distribution  $u$ satisfying 
(\ref{eq:33}) is called the \textit{Schwarz potential} of the pair
$(\Omega,\mu)$ (in \cite{Shp4} it is called "the modified Schwarz
potential"). \end{definition}

Thus whenever $\Omega$ is a quadrature domain of a measure $\mu$,
then the Schwarz potential  of the pair $(\Omega,\mu)$ exists,
moreover, in that
case the Schwarz potential $u$ is a potential, that is, $u\in
V[\Omega-\mu]$.

Therefore, in order that the existence of the Schwarz potential would
imply that $\Omega$ is a quadrature domain, it is necessary that the
Schwarz potential is a potential. But this is not sufficient.
Indeed, in \cite{KM} we presented an example of a pair $(\Omega,\mu)$
with a bounded domain $\Omega$ and a finite measure $\mu$ such that
$\mu \big |_{\R^n \setminus \Omega} = 0$ for which the Schwarz
potential of $(\Omega,\mu)$ is a potential.  But the space $
HL^1(\Omega)$ is not contained in $L^1(|\mu|)$, so identity
(\ref{eq:29}) cannot hold. Therefore there is a subtle difference
between the property of an open set $\Omega$ being a quadrature domain
of $\mu$, and the existence of the Schwarz potential of $(\Omega,\mu)$
which is a potential. Under certain restrictions, the equivalence does
hold (see e.g. \cite[\S2]{MK}, \cite{GS, Sa1}).

\begin{proposition}
 \label{prop:1}
Let $\Omega$ be an open set and assume the measure $\mu$ has a
compact support in $\Omega$ or $\mu=\sigma\lambda$, where $\sigma\in
L^\infty(\R^n)$ and $\lambda$ is the Lebesgue measure.
Then $\Omega$ is a quadrature domain of $\mu$ if and only if there
  exists $u\in V[\Omega-\mu]$ which is the Schwarz potential of
$(\Omega,\mu)$.
\end{proposition}

The main theme of this paper is null quadrature domains, that is,
identity  (\ref{eq:29}) holds with $\mu\equiv 0$. By Proposition
\ref{prop:1} this is equivalent to the existence of $u\in V[\Omega]$
satisfying (\ref{eq:33}). It means that  the domain $\Omega$ produces
zero gravitational force in the complement $\Omega^c$.  Applying the
complementary formula (\ref{eq27}) to the generalized potential of a
null quadrature domain $\Omega$, we obtain:

\begin{theorem}[\cite{Ka1, KM}]
 \label{thm:8}
An open set $\Omega$ is a null quadrature domain if and only if
$V[\R^n\setminus\Omega]$ coincides with a quadratic polynomial in $
\R^n\setminus\Omega$.
\end{theorem}

\begin{remark} Newton discovered that the gravitational attraction at
any internal
point of elliptical shell is zero (see e.g. \cite{Cha}). This  
phenomenon is
equivalent to the property that  the internal Newtonian potential of
ellipsoids is a quadratic polynomial. Conversely, if $D$ is a bounded
domain and its Newtonian potential coincides with a quadratic
polynomial in $D$, then $D$ is an ellipsoid \cite{BF, Di, Ni}. The
characterization of domains $D$ for which their  internal generalized
Newtonian potential is a quadratic polynomial is an open problem. 
\end{remark}

We discuss now local properties of the Schwarz potential. Suppose $u$
is a Schwarz potential of $(\Omega,\mu)$ and $x_0\in
\part\Omega\setminus\supp(\mu)$.  Then $u$ satisfies the free boundary
problem \begin{equation*} \left\{\begin{array}{ll}\Delta u=-1 \quad &
\text{in}\ \ \Omega\cap B_r(x_0)\\ u(x)=|\nabla u(x)|=0 \quad &
\text{on}\ \ B_r(x_0)\setminus\Omega\end{array}\right. .
\end{equation*} Lemma 2.11 in \cite{KM} claims that under these
conditions the free boundary $\part\Omega\cap B_r(x_0)$ has  zero
Lebesgue measure  (Actually this Lemma is valid in a much wider
situations and it can be applied  whenever the density of $\Omega$ is
a positive $L^\infty$-function). Therefore, if $\Omega$ is a
quadrature domain of a measure $\mu$ such that
$\supp(\mu)\cap\part\Omega=\emptyset$, then $\part\Omega$ has
zero Lebesgue
measure  and consequently $\Omega_0=\text{int}\left(
\overline{\Omega} \right)$ satisfies the same quadrature identity
(\ref{eq:29}). Therefore in the sequel we will assume that any
quadrature domain $\Omega$ is the interior of its closure, and we may
drop the requirement that  the gradient of the Schwarz
potential in (\ref{eq:33}), vanishes on $\Omega^c$.

\section{Null quadrature domains and harmonic continuation of
potentials} 
\label{sec:nfbp}

If $u$ is the Schwarz potential of $(\Omega,\mu)$ and $\R^n
\setminus
\supp(\mu)$ is
connected, then for any $v \in V[\Omega]$ the function
\begin{equation*}
 w(x):=\left\{\begin{array}{ll} v(x) \ \ \ &\text{for} \ x\in \Omega^c\\
v(x)-u(x)\ \ \ &\text{for} \ x\in \Omega \end{array}\right.
\end{equation*}
is the real-analytic (harmonic) continuation of the potential
$V[\Omega] \big |_{\Omega^c}$
across $\part\Omega$ onto $\Omega\setminus \supp(\mu)$.

In case of null quadrature domains the harmonic continuation $w$ is
an entire analytic function and the Schwarz potential satisfies 
the free boundary problem
\begin{equation}
\label{eq:15}
\left\{\begin{array}{ll}
   \Delta u =-\chi_\Omega \quad & \text{in} \
\R^n \\ u=0 \quad & \text{in} \ \R^n \setminus \overline{\Omega}
       \end{array} \right..
\end{equation}
Accordingly, if for an open set $\Omega$ the problem (\ref{eq:15}) has
a solution $u$, then we call $u$ {\it the Schwarz potential} of the
domain $\Omega$. The following  Proposition relates (\ref {eq:15}) to
the  global analytic (harmonic) continuation of the external Newtonian
potential of $\Omega$.

\begin{proposition}
\label{TM13}
 Let $\Omega = \text{Int}
(\overline{\Omega})$ be an open set in $\R^n$ and set $D :=
R^n\setminus
\overline{\Omega}$. The following properties of $\Omega$ are
equivalent: 
\begin{enumerate} 
\item[{\rm (i)}] $\Omega$ admits the Schwarz potential $u$; 

\item[{\rm (ii)}] for  any $v$ in $ V[\Omega]$, the external potential
$v
\big| _D$ extends to a harmonic function $w$ in whole space;

\item[{\rm (iii)}] for any $v^c$ in $ V[D]$, the internal potential
$v^c
\big| _D$ extends to a real analytic function $w^c$ satisfying the
Poisson equation $\Delta w^c = -1 $ in whole space. 
\end{enumerate}
Under any of these conditions, 
\begin{equation} 
\label{eq20} 
u = v - w= w^c - v^c. 
\end{equation}  
is the Schwarz potential of $\Omega$. 
\end{proposition}

\begin{proof} We prove the equivalence (i) $\Leftrightarrow$ (ii) and
the first equality in (\ref{eq20}). The remainder follows from the
complementary formula (\ref{eq27}). If $v \big| _D$ extends to an
entire harmonic function $w$, then $u := v - w$ satisfies
(\ref{eq:15}) and hence it is the Schwarz potential of $\Omega$.
Conversely, if $u$ is the Schwarz potential of $\Omega$, then it
follows from (\ref{eq:15}) that the function  $w := v-u$ is harmonic
in $\R^n$ and equals to $v$ in $D$.
\end{proof}

If the Schwarz potential of $\Omega$ is a Newtonian potential, i.e.
$u \in V[\Omega]$,
then by Proposition \ref {prop:1}, $\Omega$ is a null quadrature
domain. Thus, we have the
characterization of null quadrature domains in terms of Schwarz
potential:

\begin{corollary}[\cite{Ka1,Sa1}]
\label{TM15} An open set $\Omega$ is a null quadrature
domain
if and only if it admits the Schwarz potential $u$ with 
\begin{equation}
\label{eq21}
    u(x) = O(|x|^2\log|x|), \quad\text{as}\ x \to\infty.   
\end{equation}
\end{corollary}

We discuss now the place of null quadrature domains in a larger
class  of domains having the Schwarz potential (\ref{eq:15}). This
investigation was initiated by Shapiro in \cite {Shp2}. Note that if
$\Omega$ admits a Schwarz potential which does not satisfy
(\ref{eq21}), then $\Omega$ is not a  null quadrature domain.
Gustafsson and Shapiro \cite {Shp2}, and independently Richardson
\cite{Ri}, constructed two dimensional domains $\Omega$ with 
Schwarz potential having any prescribed singularity at infinity.
The theorem below extends their results to arbitrary dimension.

\begin{theorem}
\label{TM18} 
There exist domains $\Omega \subset \R^n ~~ (n \geq 2)$ with
the bounded complement and such that their Schwarz potential $u$
has any prescribed type of singularity at infinity, i.e. for any
$m>2$ and for any  harmonic polynomial $h_m$ of degree $m$
there are $\Omega$ with $u(x) = h_m(x) + O(|x|^2)$ as $x \to \infty$;
and there are $\Omega$ with $u$ having essential singular point
at infinity.
\end{theorem}

The methods of Gustafsson and Shapiro and of Richardson are based on a
perturbation of a disc
by special conformal mappings and hence are essentially
two-dimensional. In the proof
for the general case we use the same idea of the perturbation of a 
ball, but we reduce the existence
of $\Omega$, with the desired properties of the  Schwarz potential, to a theorem
on the existence "in the
small" (i.e. near to a given domain $\Omega_0$) of the solution to the inverse problem of
potential. The first existence theorem of this type was proved by
Sretenskii \cite{Sr1} for
the external potential and for $\Omega_0$ being a ball. Ivanov \cite{Iv1} extended the theorem
of Sretenskii to starlike domains by a modification of the techniques of Lichtenstein from
the equilibrium figures theory. For $n=2$ the corresponding results due to Cherednichenko
(see references in \cite{Ch}). Further generalizations were obtained by
Prilepko and Isakov \cite {Pr, Is}, and they  rely on Ivanov's arguments. The
existence theorem below for
the interior inverse problem of potential is a special case of Corollary 5.1.3
from \cite{Is}.
\begin{theorem}\label{TM19} Let $D_0$ be a bounded domain in $\R^n ~~
(n \geq 2)$ with
Newtonian potential $w_0$ and with $\part D_0 \in C^{2, \lambda}$. Fix arbitrarily $\eps > 0$
and denote by $L_{D_0, \eps}$ the set of all functions $w$ satisfying
the  Poisson equation 
\[
    \Delta w = -1 \quad \text{in} \quad D_{\eps} :=\{x :
\text{dist}(x, D_0) < \eps \}.
\]
Then for any $r>0$, there is $\delta>0$ such that for $w \in L_{D_0,
\eps}$ satisfying the condition
\begin{equation}
\label{eq24}
    |w(x)-w_0(x)| + |\nabla w(x) - \nabla w_0(x)| < \delta
 \quad \text{on} \quad \part D_0,
\end{equation}
there is a  domain $D$, diffeomorphic to $D_0$, having the regular analytic
boundary lying
in the $r-$neighborhood of $\part D_0$ and such that the interior Newtonian potential of $D$ 
equals to $w \big | _D$. There
exists $r_0 = r_0(D_0, \eps)$ such that for $r<r_0$ the domain $D$ satisfying
the above conditions is unique.
\end{theorem}

{\it Proof of Theorem \ref {TM18}}. 
Let $D_0 = B_1(0)$ be the unit
ball. The interior potential of $D_0$ is given by the known formula 
\[w_0(x) = c_n - \frac {|x|^2} {2n} \] 
where $c_2 = 1/4$ and $c_n = 1
/2(n-2)$ for $n>2$.
Let $\eps, r >0$, and choose $\delta$ by Theorem
\ref{TM19}. Take an arbitrary entire harmonic function $h$ and set 
\[M := \max _{|x|=1} ~ (|h(x)| + |\nabla h(x)|). \] 
Then the function 
\[ w(x) := c_n - \frac {|x|^2} {2n} + \frac {\delta}{2M} h(x) \]
is a solution of Poisson equation $\Delta w = -1$ in whole space,
satisfying (\ref{eq24}), so by Theorem \ref {TM19} there is a domain
$D$ diffeomorphic to the ball with the interior potential $w$. Denote
by $v$ the Newtonian potential of $D$. By Proposition \ref{TM13} we
conclude that the function $u := w-v$ is the Schwarz potential of
$\Omega := \R^n \setminus \overline{D} ~$. If $h$ is a harmonic
polynomial of degree $m>2$, then
$u(x) = {\delta \over 2M}h(x) + O(|x|^2)$
as $x \to \infty ~$. If $h$ is a non-polynomial entire harmonic function,
then $u$ has an essential singular point at infinity. $\square$

\begin{remark}[Shapiro \cite{Shp2}]
\label{rm3} 
If $\Omega_k$ is a domain in $\R^k ~~(k < n)$
chosen by Theorem \ref{TM18}, then the Schwarz potential of the
cylinder $\Omega := \Omega_k\times \R^{n-k}$ has the same properties
as that of $\Omega_k$, but $\Omega$ has unbounded complement.
\end{remark}

The following result of Shapiro \cite{Shp2} (see also Shahgholian
\cite{Shg}) shows that if
the complement of $\Omega$ is slightly larger  than any cylinder over a
bounded domain at 
infinity, and if the Schwarz potential of $\Omega$ is a
 {\it tempered distribution},
then $\Omega$ is a null quadrature domain.

\begin{proposition}\label{TM20} Suppose $\Omega$ admits the Schwarz 
potential $u$, and
\begin{enumerate}
    \item[{\rm (i)}] $\R^n \setminus \overline{\Omega}$ contains 
balls of arbitrary radius,
    \item[{\rm (ii)}] $u$ is a tempered distribution.
\end{enumerate}
\noindent Then $\Omega$ is a null quadrature domain.
\end{proposition}

Shapiro \cite{Shp2} showed that condition (ii) in Proposition 
\ref{TM20} is indispensable.  He constructed an example where $\Omega$
and
$\Omega^c$ are ``equally large'', that is, both contain half-spaces.
\begin{proposition} \label{TM21} There exist domains 
$\Omega= \{(x_1,...,x_{n-1},x_n)\in \R^n : x_n < f(x_1,...,x_{n-1})\}$
such that
\begin{enumerate}
    \item[{\rm (i)}] $f$ is a bounded real analytic function in $\R^{n-1}$,
    \item[{\rm (ii)}] $\Omega$ admits the Schwarz potential that is not a tempered distribution.
\end{enumerate}
\end{proposition}

\section{Geometric properties of null quadrature domains}
\label{sec:gpnqd} 

In this section we prove that the complement of a null quadrature domain is
a convex set with analytic boundary. This result was conjectured by H.S. Shapiro
and proved by him under some additional conditions \cite{Shp2}
(see also \cite{CKS} and \cite{Ou}).

The main result of this section is the following:
\begin{theorem}
 \label{thm:1} Let $\Omega$ be a null quadrature domain with Schwarz 
potential $u$.
 Then
\begin{enumerate}
 \item[{\rm (i)}] there is a positive constant $C$ such that
 \[
    |u(x)|\leq C(1+|x|)^2, \quad x \in \R^n;
\]
 \item[{\rm (ii)}] $u \leq 0$ in $\R^n$;
 \item[{\rm (iii)}] $\part^2 u/\part e^2\leq 0$ in $\R^n$ for any unit
vector $e$;
  \item[{\rm (iv)}] the complement of $\Omega$ is a convex set and
$\part \Omega$ is an analytic surface.
\end{enumerate}
\end{theorem}

Properties (ii)-(iv) were proved in \cite[Theorem II]{CKS} under the
restriction that the Schwarz potential $ u$ has a quadratic growth.
According to Corollary \ref{TM15},  the Schwarz potential of a null
quadrature domain has $|x|^2\log |x|$ growth at infinity.  
Thus the principal technical target of this section is the
establishment 
the quadratic growth.

Note that Theorem \ref{thm:1} (ii) solves a
conjecture which was raised by Shapiro \cite{Shp2} and Shahgholian
\cite{Shg}, namely, that any solution $u$ to (\ref{eq:15})
satisfying  the growth condition $u(x)=O(|x|^2\log |x|)$ is
non-positive. 

When the Schwarz potential $ u$ is  a limiting
(global) case of the classical obstacle problem, then   $u\leq0$ in
$\R^n$ in addition to (\ref{eq:15}). In this case
the quadratic growth is a consequence of the boundedness of the second
order derivatives and   properties  (iii) and (iv) follow from the
known regularity
results for the obstacle problem \cite{Ca1,Ca2},
\cite[Ch.2]{Fr}.

The proof of Theorem \ref{thm:1} will be split into two cases:
``thick" and ``thin" complement $(\R^n\setminus\Omega)$ at infinity,
where the ``size" is expressed here by Lebesgue  density at
infinity.

 In the case of thick complement we will prove the expected
result, namely, $\Omega$ is a half-space. This means that the boundary
$\part\Omega$ is regular at infinity. We will employ common 
techniques of free boundaries, such as blow-up and a monotonicity
formula  combined with Shapiro's quasi-balayage \cite{Shp5}. In the
singular case of thin complement we will adopt  special techniques
from \cite{KM}.

\subsection{Null Quadrature Domains with Thick Complement}
\label{sec:Thick Complement} $~$

We denote by $|D|$ the Lebesgue measure of a set $D$ in $\R^n$ and
$B_\rho$  the ball with radius $\rho$ centered at the origin.
\begin{theorem}
 \label{thm:3}
A null quadrature domain $\Omega$ satisfying condition
\begin{equation}
\label{eq:13}
    \limsup_{\rho\to \infty}\frac{|\Omega^c\cap B_\rho|}{|B_\rho|}>0
\end{equation} 
is a half-space.
\end{theorem}

\begin{lemma}
 \label{lem:quadratic growth1}
Let $u$ be the Schwarz potential of null quadrature domain $\Omega$
which satisfies condition (\ref{eq:13}),
then there is a positive constant $C$ such that
\begin{equation}
 \label{eq:2}
 |u(x)|\leq C\left( 1+|x| \right)^2.
\end{equation}  
\end{lemma}

 Set  
\begin{equation}
 \label{eq:3}
S_j=S_j(u)=\sup_{B_{2^j}} |u|, \quad j=1,2,3,... 
\end{equation}
Then (\ref{eq:2}) is equivalent to 
\begin{equation}
 \label{eq:4} 
S_j\leq C\left(2^{j}\right)^2\quad \text{ for all} \ j\in \mathbb{N}.
\end{equation} 

 We  use first Shapiro technique's of quasi-balayage \cite{Shp5} which
gives (\ref{eq:4}) along a subsequence of $\{2^j\}$, and then the  main
difficulty
 is to extend it to the  entire  sequence.

\begin{theorem}[Shapiro {\cite[\S 6]{Shp5}}]
\label{thm:2}
 Let $\sigma\in L^\infty(\R^n)$ and $w$ be a  tempered distribution
on $\R^n$ satisfying the system 
\begin{equation}
 \left\{\begin{array}{ll}\Delta w =-\sigma\chi_{\Omega} \quad &
\text{in}\ \R^n\\ w(x)=0 \quad & x\in \R^n\setminus \overline{\Omega}.
\end{array}\right.
\end{equation}  
If 
\begin{equation}
\label{eq:14}
 \frac{|\Omega^c\cap B_\rho|}{|B_\rho|}\geq c_0>0,
\end{equation} 
then for  $|x|=\rho$ there holds
\begin{equation}
\label{eq:18}
 |w(x)|\leq C_1\rho^2,
\end{equation} 
and the constant $C_1$ depends on $c_0$, $\|\sigma\|_{L^\infty}$ and
the dimension.
\end{theorem}

\begin{proposition}
Retaining  the hypotheses of Lemma \ref{lem:quadratic growth1},
then there is an infinite subset $\mathbb{K}$ of the natural numbers 
such that  (\ref{eq:4}) holds for $ j_k\in\mathbb{K}$.
\end{proposition}

\begin{proof}
Assuming (\ref{eq:13}), then
there is a sequence $\{\rho_k\}$ tending to infinity and a positive
constant $c_0$  such that (\ref{eq:14}) holds. 
We may apply now  Theorem \ref{thm:2} to $u$, the Schwarz
potential of $\Omega$, then inequality (\ref{eq:18}) together with
the fact that $|u|$ is subharmonic yields that
\begin{equation*}
 \sup_{B_{\rho_k}}|u|=\sup_{\{|x|=\rho_k\}}|u|\leq C_1\rho_k^2.
\end{equation*}
Since for each $k$ there is an integer
$j_k$ such that $2^{j_k}\leq \rho_k\leq 2^{j_k+1}$, we get 
\begin{equation}
 S_{j_k}\leq 4C_1\left( 2^{j_k} \right)^2
\end{equation} 
for infinitely many $j_k$.
\end{proof}

Thus Lemma \ref{lem:quadratic growth1} follows from:
\begin{lemma}
 \label{lem:2} Let $u$ be the Schwarz potential of null quadrature
domain $\Omega$ and  defined $S_j$  by (\ref{eq:3}).  If there
is an infinite subset $\mathbb{K}$ of the
natural numbers such that
\begin{equation}
\label{eq:8}
 S_j\leq C_2\left({2^j}\right)^2 \quad \text{for} \ j\in \mathbb{K},
\end{equation} 
then there is a constant $C_0$ so that 
\begin{equation}
\label{eq:5}
 S_j\leq 4C_0\left(2^{j}\right)^2\quad \text{ for all positive}\ j.
\end{equation} 
\end{lemma}

In the proof of this lemma we  will use  Alt, Caffarelli and
Friedman's monotonicity formula
\cite{ACF}.
\begin{theorem}[\cite{ACF}]
\label{thm:ACF} 
Suppose $v_1$ and
$v_2$ are non-negative subharmonic functions in $B_R$ such that
$v_1(x)v_2(x)=0 $ in $B_R$ and $v_1(0)=v_2(0)=0$. Then for $0<r<R$,
\begin{equation}
\Phi(r):=\Phi(r,v_1,v_2)=\frac 1 {r^4}\left(
\int\limits_{B_r}\frac{|\nabla v_1|^2}{|x|^{n-2}}dx \right)\left(
\int \limits_{B_r}\frac{|\nabla v_2|^2}{|x|^{n-2}}dx \right).
\end{equation} 
 is a nondecreasing function of $r$.
\end{theorem}


Let us make several remarks about the applications of this theorem 
here. We may assume the origin belongs to $\Omega^c$. Then the
partial derivatives $\part_i u$, of the Schwarz potential $u $ of
$\Omega$, are harmonic in $\Omega$ and vanish on
$\R^n\setminus\Omega$. Therefore  $\part_i u^+:=\max\{\part_i
u,0\}$ and $\part_i u^-:=-\min\{\part_i u,0\}$ are subharmonic
and satisfy the requirements of Theorem \ref{thm:ACF}. Further
properties are:
\begin{enumerate}
 \item[a)] \label{prop:a} 
 If the second derivatives of $u$ are bounded by  $K$ in
$B_r$, then 
\begin{equation*}
 \Phi(r,\part_i u^+,\part_i u^-)\leq
\frac{(nK)^4\omega_n^2}{4}.
\end{equation*} 
 
\item[b)] $\Phi$ is homogeneous with respect to the scaling
  $u_\rho(x)=u(\rho x)/\rho^2$, that is,
\begin{equation}
 \Phi(r\rho,\part_i u^+,\part_i u^-)=\Phi(r,\part_i
u_\rho^+,\part_i u_\rho^-).
\end{equation} 
\end{enumerate}

\noindent \textit{Proof of Lemma \ref{lem:2}.}  
Let $\mathbb{M}$ be the set of all
positive integers such that
\begin{equation}
\label{eq:10}
 4S_{j-1}>S_{j}.
\end{equation}

We consider first the case when $\mathbb{M}$  is finite.
 Then there is  $j_0\in \mathbb{K}$ such
that
\begin{equation}
\label{eq:6}
 S_j\leq \frac 1 4 S_{j+1}\quad \text{for all}\ j\geq j_0.
\end{equation} 
 Since $\mathbb{K}$ is infinite,
for each such $j$ there is $j_k\in
\mathbb{K}$ such that $j<j_k$. Let $m=j_k-j$, then by (\ref{eq:8})
and (\ref{eq:6}), 
\begin{equation*}
 S_j\leq\frac {S_{j+m}} {4^m} =\frac{S_{j_k}}{4^m}\leq
C_2 \frac{ \left(2^{(j+m)}\right)^2}{4^m}=C_2 \left(2^{j}\right)^2
\quad \text{for all}\ j\geq j_0.
\end{equation*}
 For $j<j_0$,
$S_j\leq S_{j_0}\leq C_2 2^{2(j_0-j)}\left(2^{j}\right)^2$. Thus     
(\ref{eq:5}) holds for all $j$.

In the case where $\mathbb{M}$ is infinite, we will prove  that there
is a positive
constant $C_0$ such that 
\begin{equation}
\label{eq:7}
 S_j\leq C_0\left(2^j  \right)^2\quad \text{for all}\ j\in \mathbb{M}.
\end{equation} 
Then inequality (\ref{eq:5}) follows from (\ref{eq:7}).  Indeed, if
not, then there
is $j_1$ such that $S_{j_1}>4C_0\left( 2^{j_1}
\right)^2$. Since $\mathbb{M}$ is infinite, there is an  integer
$j_2\in \mathbb{M}$ such that $j_1<j_2$. Let
\begin{equation*}
 j_3=\max\{j<j_2: 4C_0\left( 2^{j}\right)^2<S_{j}\}.
\end{equation*}  
Then
\begin{equation*}
 S_{(j_3+1)}\leq 4\left( 4C_0\left(2^{j_3}\right)^2 \right)<4S_{j_3}.
\end{equation*} 
Hence $j_3+1\in \mathbb{M}$, so by (\ref{eq:7}),
\begin{equation*}
 4C_0\left(2^{j_3}\right)^2<S_{j_3}\leq S_{(j_3+1)}\leq
C_0\left(2^{(j_3+1)}\right)^2= 4C_0\left(2^{j_3}\right)^2
\end{equation*} 
and this yields a contradiction.

Thus it remains to prove inequality (\ref{eq:7}). We will assume it
does not hold and derive a contradiction. 
So assume there is  infinite sequence $\{j_m\}\subset\mathbb{M}$
such that
\begin{equation}
 \label{eq:11}
 S_{j_m}\geq m \left(2^{j_m}\right)^2
\end{equation} 
 and set 
\begin{equation*}
u_{m}(x)=\frac{u(2^{j_m}x)}{S_{j_m-1}}.                               
\end{equation*}
This sequence has the following properties: Since $0\in\Omega^c$,
$u_m(0)=|\nabla
u_m(0)|=0$ for all $m$; 
\begin{equation}
\label{eq:9}
 \sup_{B_{ (1/ 2)}}| u_m|=1;
\end{equation} 
from (\ref{eq:10}), 
\begin{equation}
 \label{eq:qagr:5} 
 \sup_{B_1}|u_m|=\frac{S_{j_m}}{S_{j_m-1}}\leq 4
\end{equation}
and from (\ref{eq:11})
\begin{equation}
 \label{eq:qagr:6} 
|\Delta u_m(x)|=|\frac{2^{2j_m}}{S_{j_m-1}}\chi_\Omega(2^{j_m}x)|\leq
\frac{4}{m}.
\end{equation} 
Applying elliptic estimates, we get that
$\{u_m\}$ is bounded in $C^{1,\alpha}\left(B_{(3/4)}\right)$ and
therefore exists a function $u_\infty$ such that $u_m\to u_\infty$ and
$\nabla u_m\to \nabla u_\infty$
uniformly in $B_{(1/2)}$ (we do not distinguish between  sequences 
and subsequences).

We  invoke now Theorem \ref{thm:ACF}. Since $\mathbb{K}$ is infinite,
for each $j_m\in\mathbb{M}$ there is $j_k\in \mathbb{K}$ such that
$j_m\leq j_k$. So
by the  homogeneity  and the monotonicity of $\Phi$ we have,
\begin{equation}
\label{eq:12}
\begin{split}
& \left(\frac{S_{j_m-1}}{\left(
2^{j_m} \right)^2}\right)^4\Phi(\frac 1 2, \part_i
u^+_m,\part_i u^-_m)\\ = &\Phi(\frac 1 2 2^{j_m}, \part_i
u^+,\part_i u^-)\leq\Phi(\frac 1 2 2^{j_k}, \part_i u^+,\part_i
u^-).
\end{split}
\end{equation}
In order to bound $\Phi(\frac 1 2 2^{j_k}, \part_i u^+,\part_i
u^-)$ we set
\begin{equation*}
 u_k(x)=\frac{u(2^{j_k}x)}{\left(2^{j_k}\right)^2},
\end{equation*}  
when  $j_k\in \mathbb{K}$.
Then $u_k$ satisfies the equations
\begin{equation}
\label{eq:qagr:7} 
 \left\{\begin{array}{ll} 
\Delta u_k =1, \quad &\text{in}\ \Omega_k \cap B_{1}\\
u_k(x)=|\nabla u_k(x)|=0, \quad &
\text{on} \ B_{1}\setminus \Omega_k\end{array}\right.,
\end{equation} 
where $\chi_{\Omega_k}(x)=\chi_\Omega(2^{j_k}x)$. So  by condition
(\ref{eq:8}), $\sup_{B_1}|u_k|\leq C_2$. Therefore \cite[Theorem I]{CKS} yields that
$\|u_k\|_{C^{1,1}(B_{(1/2)})}\leq
C$ and the constant $C$ is independent of $k$. Using  properties   a)
 and b)  below Theorem \ref{thm:ACF}, we get that
\begin{equation*} 
\Phi(\frac 1 2
2^{j_{k}},\part_i u^+,\part_i u^-)=\Phi(\frac 1 2,\part_i
u^+_k,\part_i u^-_k)\leq C.
\end{equation*} 
The  combination of the above inequality with (\ref{eq:10}),
(\ref{eq:11}) and
(\ref{eq:12}) implies that
\begin{equation*}
\begin{split}
 & \left( \frac 1 2 \right)^{2n} \left( \int
\limits_{B_{(1/2)}}|\nabla
\part_i
u_m^+|^2 dx\right)\left( \int \limits_{B_{(1/2)}}|\nabla \part_i
u_m^-|^2 dx\right) \\ &\leq \Phi(\frac 1 2, \part_i
u^+_m,\part_i
u^-_m)\leq \left(\frac{\left( 2^{j_m}
\right)^2}{S_{j_m-1}}\right)^4\Phi(\frac 1 2
2^{j_{k}},\part_i u^+,\part_i u^-)\\ &\leq 
C \left(\frac{\left( 2^{j_m}
\right)^2}{S_{j_m-1}}\right)^4
\leq
C\left(\frac 1 m\frac{S_{j_m}}{S_{j_m-1}}\right)^4 \leq
C \left(\frac 4 m\right)^4.
\end{split}
\end{equation*} 
Let $M^+_m$ and $M^-_m$ denote the mean of $\part_i u^+_m$ and
$\part_i u^-_m$ over the ball $B_{(1/2)}$, respectively. Then
Poincar\'e inequality implies 
\begin{equation*}
\begin{split}
& \left( \int \limits_{B_{(1/2)}}| \part_i
u_m^+-M_m^+|^2 dx\right)\left( \int \limits_{B_{(1/2)}}| \part_i
u_m^- - M_m^-|^2 dx\right) \\ \leq C &\left( \int
\limits_{B_{(1/2)}}|\nabla
\part_i
u_m^+|^2 dx\right)\left( \int \limits_{B_{(1/2)}}|\nabla \part_i
u_m^-|^2 dx\right) \leq C \left(\frac 4 m\right)^4.
\end{split}
\end{equation*} 
Letting $m$ tends to infinity, we get
\begin{equation}
\label{eq:17} 
 \left( \int \limits_{B_{(1/2)}}| \part_i
u_\infty^+-M_\infty^+|^2 dx\right)\left( \int \limits_{B_{(1/2)}}|
\part_i
u_\infty^- - M_\infty^-|^2 dx\right) =0.
\end{equation}
Since $\part_i u_\infty^\pm(0)=0$,  equality (\ref{eq:17}) implies
that
$\part_i u_\infty$ does not change sign in $B_{(1/2)}$. In addition,
(\ref{eq:qagr:6}) shows that $\part_i u_\infty$ is harmonic in
$B_{(1/2)}$, so the maximum principle implies that $\part_i
u_\infty\equiv 0$ for $i=1,...,n$. Thus $u_\infty$ 
identically equals to a constant and since $u_\infty(0)=0$, we
conclude that
$u_\infty\equiv 0$ in $B_{(1/2)}$. But this contradicts
(\ref{eq:9}). This completes the proof of Lemma \ref{lem:2}. 
\hfill{$\square$}

\vskip 5mm

Having the quadratic growth  at hand (Lemma \ref{lem:quadratic
growth1}),  we can now follow  the proof of 
\cite[Theorem II, Case 2]{CKS} and conclude that if $\Omega$ is a null
quadrature domain satisfying  condition (\ref{eq:13}), then it is a
half-space. But we
have
chosen here a different approach which is based upon the computation
of a potential of
a cone (\ref{eq:30}). This method  requires a precondition that the Schwarz
potential is non-positive, therefore   in the case of a global 
obstacle problem
we may use it without relying on \cite{CKS}.

\vskip 5mm
\noindent \textit{Proof of Theorem \ref{thm:3}.}
The Schwarz potential $u$ is non-positive since it has quadratic     
expression growth, therefore by
\cite[Corollary 7]{Ca2}, $\R^n\setminus\Omega$ is convex.  
Let $\{\rho_k\}$ be a sequence tending to infinity and set 
\begin{equation*}
 u_k(x):=\frac{u(\rho_k x)}{\rho_k^2}.
\end{equation*}
Then inequality (\ref{eq:2}) implies that $u_k$ converges to a function $u_\infty$
and $\nabla u_k$ converges to $\nabla u_\infty$ uniformly on compact subset of
$\mathbb{R}^n$. Let $K:=\{x: u_\infty(x)<0\}$, then $\Delta u_\infty =-1$ on $K$,
$u_\infty$ vanishes on complement $K^c$ and $K^c$ is also convex. Without loss of
generality we may assume the origin belongs to $\part\Omega$ which implies that $0\in
\part K$. By means of \cite[Theorem 2.5]{KS}, $K$ is a cone and therefore  
$u_\infty$ is a potential of the cone $K$. So $u_\infty$  has the formula
(\ref{eq:30}) and since $u_\infty\leq0$ in $\R^n$, the expression 
(\ref{eq:30}) 
cannot contain the logarithmic term.  In addition $K^c$ is convex,   so Corollary 3.9
in \cite{KM} implies that $K$ is half-space. 

We will accomplish the proof by showing that 
\begin{equation}
\label{eq:25}
 K^c\subset \Omega^c.
\end{equation} 
Since $\Omega^c$ is convex and $K^c$ is a half-space, the above
inclusion implies that $\Omega$  must be a half space. Note
that both $K^c$ and $\Omega^c$ are closed sets, hence it suffices to
show ${\rm int}(K^c)\subset{\rm int}(\Omega^c)$. Suppose it fails,
then there is $x_0\in {\rm int}(K^c)$ such that $x_0\in
\overline{\Omega}$. Since $\Omega^c$ is convex, $\rho x_0\in
\overline{\Omega}$ for all $\rho\geq 1$.  Let $r$ be a positive
number so that $B_r(x_0)\subset {\rm int}(K^c)$, then by Caffarelli's
Lemma  (see e.g. \cite[Lemma 5]{Ca2}),
\begin{equation}
 \inf_{B_{r\rho}(\rho x_0)}u(x)\leq -c \left( r\rho \right)^2
\end{equation} 
for any $\rho\geq 1$ and where $c$ is a positive constant. Hence 
\begin{math}
 \inf_{B_{r}( x_0)}u_k(x)\leq -c r ^2
\end{math} 
for all $k$. Since $u_k$ converges uniformly to $u_\infty$, we have
obtained that $B_r(x_0)\not\subset {\rm int}(K^c)$ which is a
contradiction. This establish (\ref{eq:25}) and completes the proof
of Theorem \ref{thm:3}.
\hfill{$\square$}

\subsection{Null Quadrature Domains with Thin Complement} $~$

In this section we will prove the quadratic growth of the Schwarz potential of null
quadrature $\Omega$ under the complementary condition to
(\ref{eq:13}), namely
\begin{equation}
\label{eq:34}
    \lim_{\rho\to \infty}\frac{|\Omega^c\cap B_\rho|}{|B_\rho|}=0.
\end{equation} 
The main
idea here is to show that if a set $D$ has  zero Lebesgue
density at infinity, then the blow-up of its potential will tend to
zero. Using the representation (\ref{eq20}) of the Schwarz potential, we
will get that the blow-up of
the Schwarz potential tends to a quadratic polynomial and that yields
the quadratic growth in this case.  

\begin{proposition}
\label{prop:2.3}
 If $D$ is set satisfying the condition
\begin{equation}
\label{eq:19}
    \lim_{\rho\to \infty}\frac{|D\cap B_\rho|}{|B_\rho|}=0,
\end{equation}
then
\begin{equation*}
 \lim_{\rho\to\infty}\|\chi_D(\rho\ \cdot)\|_{\mathcal{L}}=0,
\end{equation*}
where the norm $\|\cdot\|_{\mathcal{L}}$ is defined by (\ref{eq:1}). 
\end{proposition}

\begin{proof}
 Making  a change of the variables, we have
\begin{equation}
\label{eq:26} 
 \int \limits_{B_R}\chi_D(\rho x)dx=R^n\frac{|B_{(\rho R)}\cap
D|}{|B_{(\rho
R)}|}\to 0\quad \text{as}\ \rho\to \infty.
\end{equation}
For a given $\epsilon>0$, we may find $R>0$ such that
\begin{equation*}
 \int \limits_{\{R\leq |x|\}}\frac{\chi_D(\rho x)}{\left( 1+|x|
\right)^{n+1}}dx<\epsilon
\end{equation*}
for all $\rho$. Hence,
\begin{equation*}
\begin{split}
 \|\chi_D(\rho\
\cdot)\|_{\mathcal{L}} &=\int \limits_{B_R}\frac{\chi_D(\rho
x)}{\left( 1+|x|\right)^{n+1}}dx+\int \limits_{\{R\leq
|x|\}}\frac{\chi_D(\rho
x)}{\left( 1+|x|\right)^{n+1}}dx
\\ & \leq \int \limits_{B_R}\chi_D(\rho
x)dx+\epsilon\to\epsilon \quad\text{as}\ \rho\to\infty.
\end{split}
\end{equation*}
\end{proof}

\begin{proposition}
\label{prop:2.4} 
Let $V_2(D)$ be the potential that is defined by (\ref{eq:31}). If a
set $D$ satisfies condition (\ref{eq:19}), then
\begin{equation*}
\lim_{\rho\to\infty}\frac{1}{\rho^2}V_2\left(D\right)(\rho x)=0
\end{equation*}
uniformly on compact subsets of $\R^n$. 
\end{proposition}

\begin{proof}
Let $\phi_\rho$ be a cut-off function such that $\phi_\rho(x)=1$ for $|x|\leq \rho$,
$\phi_\rho(x)=0$ for $|x|\geq 2\rho$ and $|\part^\alpha \phi_\rho(x)|\leq C 
\rho^{-|\alpha|}$, and  set \begin{equation*}
w_\rho(x):=\phi_\rho(x)\frac{1}{\rho^2}V_2\left(D\right)(\rho x).
\end{equation*}
 Since $ w_\rho\in C_0^1(\mathbb{R}^n)$, $V_2\left( \Delta w_\rho
\right)=w_p$ and hence it suffices to show that $V_2\left( \Delta
w_\rho\right)$ tends to zero uniformly. We first claim that $\|\Delta
w_\rho\|_{\mathcal{L}}\to 0$, indeed,
\begin{equation*}
\begin{split}
 \Delta w_\rho(x) &=\phi_\rho(x)\chi_D(\rho
x)\\ &+\frac{2}{\rho}\left\{\nabla
\phi_\rho(x)\cdot \nabla V_2(D)(\rho x)\right\}
+\frac{1}{\rho^2}\Delta \phi_\rho(x)V_2(D)(\rho x).
\end{split}
\end{equation*} 
The $\mathcal{L}$-norm of the first term tends to zero by Proposition
\ref{prop:2.3} . By a virtue of the growth properties of the
 potential (\ref{eq26}) and the properties of the mollifier $\phi_\rho$ , 
\begin{equation*}
\begin{split}
 &|\frac{2}{\rho}\left\{\nabla
\phi_\rho(x)\cdot \nabla V_2(\chi_D)(\rho x)\right\}| \\ \leq C &
\frac{2\chi_{\{\rho\leq |x|\leq 2\rho\}}}{\rho^2}(1+|\rho
x|)\log(2+|\rho x|)\leq C
\log(2+\rho^2)
\end{split}
\end{equation*}
and
\begin{equation*}
\begin{split}
& |\frac{1}{\rho^2}\Delta
\phi_\rho(x)V_2(\chi_D)(\rho x)|\\ \leq C &
\frac{\chi_{\{\rho\leq |x|\leq 2\rho\}}}{\rho^4}(1+|\rho
x|)^2\log(2+|\rho x|)\leq C
\log(2+\rho^2).
\end{split}
\end{equation*}
Thus,
\begin{equation*}
\begin{split}
  & \left\|\frac{2}{\rho}\left\{\nabla
\phi_\rho\cdot \nabla V_2(\chi_D)(\rho \ \cdot)\right\}
+\frac{1}{\rho^2}\Delta
\phi_\rho V_2(\chi_D)(\rho \ \cdot)\right\|_{\mathcal{L}} \\ \leq & C
 \log(2+\rho^2) \int \limits_{\{\rho\leq |x|\leq
2\rho\}}\frac{1}{\left(1+|x|\right)^{(n+1)}}dx\leq C\frac{
\log(2+\rho^2)}{\rho}\to0. 
\end{split}
\end{equation*} 

Next we proceed in a similar manner to the proof of Theorem 3.12 in
\cite{KM}. Let $J_2(x,y) $ be the kernel that is defined by
(\ref{eq29}). Then for  $|x|\leq R$, 
\begin{equation*}
 \begin{split}
 &|V_2(\Delta w_\rho)(x)| \leq  \int \limits_{\{|y|\leq
1\}}J(x-y)\Delta
w_\rho(y)dy \\ + & \int \limits_{\{1\leq |y|\leq
2R+1\}}|J_2(x,y)\Delta
w_\rho(y)dy| +\int \limits_{\{2R+1 \leq |y|\}}|J_2(x,y)\Delta
w_\rho(y)dy|.
 \end{split}
\end{equation*} 
Since $D$ has zero Lebesgue density, we see from (\ref{eq:26}) that 
$\Delta w_\rho$ tends to zero almost everywhere
in $B_R$,  so the first two terms tend to zero by Lebesgue dominated
theorem, while for the last one  we  use the estimate
\begin{equation*}
 |J_2(x,y)|\leq C \frac{R^3}{|y|^{(n+1)}}\quad \text{for}\  |x|\leq R\
\text{and}\ 2R+1\leq |y|.
\end{equation*}
Knowing that $\|\Delta w_\rho\|_{\mathcal{L}}\to 0$, we get that 
\begin{equation*}
 \int \limits_{\{2R+1 \leq |y|\}}|J_2(x,y)\Delta w_\rho(y)dy|\leq C
R^3\int
\frac{|\Delta w_\rho(y)|dy}{\left( 1+|y| \right)^{(n+1)}}=
C R^3\left\|\Delta w_\rho\right\|_{\mathcal{L}}
\end{equation*}
also tends to zero.
\end{proof}
Let  $u$ be the Schwarz potential of a null quadrature domain
$\Omega$ then by formula (\ref{eq20}),
\begin{equation}
 u(x)=P(x)-V_2({\Omega^c})(x),
\end{equation} 
where $P=w^c_{\mid_{\Omega^c}}$ for some  $w^c\in V[\Omega^c]$.
According to Theorem \ref{thm:8},
\begin{equation}
\label{eq:32}
 P(x)=\sum_{i,j}a_{ij}x_ix_j+\sum_i b_ix_i+c.
\end{equation} 

Applying Proposition \ref{prop:2.4}  to $\Omega^c$ we get the desired
quadratic growth. 
\begin{lemma}
 \label{lem:quadratic growth:2}
Suppose $\Omega$ is a null quadrature domain such that the complement
$\Omega^c$ satisfies condition (\ref{eq:19}), then 
\begin{equation}
\label{eq:20}
 \lim_{\rho\to\infty} \frac{u(\rho x)}{\rho^2}=\sum_{i,j}a_{ij}x_ix_j
\end{equation} 
and the convergence is uniform on compact subsets
of $\mathbb{R}^n$. Consequentially, there  is a constant $C$ such that
\begin{equation*}
 |u(x)|\leq C (1+|x|)^2.
\end{equation*}
\end{lemma}

\begin{remark}
 It is possible to show the quadratic growth for unbounded
quadrature domains  of measures with compact support. In that
case the monotonicity formula Theorem \ref{thm:ACF} need to be
replaced by a monotonicity formula of Caffarelli, Jerison and Kenig
\cite{CJK}. 
\end{remark}

\subsection{The structure of null quadrature domains} $~$

In this subsection we discuss the relations between   shape of
null quadrature domains and the quadratic polynomial (\ref{eq:32}),
the internal potential of the complement. 

We first note that if the complement $\Omega^c$ of a null quadrature
domain has  positive upper Lebesgue density at infinity, then Theorem
\ref {thm:3} implies that $\Omega $ is a half space. Hence, in an
appropriate coordinate system $P(x)=-\frac 1 2x_n^2$.

In the case of thin complement we conclude from the limit
(\ref{eq:20}) and  Theorem \ref{thm:1} (ii) that:
\begin{corollary}
 Suppose $\Omega$ is a null quadrature domain such that the complement
$\Omega^c$ satisfies condition (\ref{eq:19}), then 
\begin{equation*}
 \sum_{i,j}a_{ij}x_ix_j\leq 0.
\end{equation*}
\end{corollary}

Let $A=(a_{ij})$ be the matrix associated with this polynomial. The
limit (\ref{eq:20}) identifies  with each null quadrature domain with
thin complement a unique matrix $A$.

When $\Omega^c$ is bounded, then it is an ellipsoid and
consequently 
the matrix $A$ is negative definite  (see. e.g. \cite{Ka3}). In
addition, up to a translation and dilation, 
there is one to one corresponding between the class of null quadrature
domains having a bounded complement and negative defined matrices. 

In case $\Omega^c$ is unbounded, it must contain an infinite ray,
since  by Theorem \ref{thm:1} it is a convex set. We
may
choose a coordinates system so that this ray is the positive
$x_1$-axis. Applying the limit  (\ref{eq:20}) and its first order
derivatives along this ray, we conclude that $a_{1,j}=0$,
$j=1,...,n$. Thus the matrix $A$ is rank deficient when $\Omega^c$ is
unbounded.

We use now the following known geometric facts (for the proof see
e.g. Appendix of \cite{Al})

\begin{lemma}
\label{TM4}
Let $D \neq \R^n$ be an unbounded convex domain in $\R^n$. Then
$D$ has (at least) one of the following properties:
  \par {\rm(i)} $D$ is a cylinder;
  \par {\rm(ii)} $\part D$ is a graph of a convex function in an
appropriate Cartesian
coordinate system in $\R^n$.
\end{lemma}

So in the case of unbounded thin complement there are two options for
each semi-negative symmetric matrix $A$ with rank $k<n$:

\begin{enumerate}
 \item[(a)] If the complement is contained in a strip, then it follows
from the forthcoming paper \cite{MK} (see also \cite[Theorem
4.13]{KM}) that $\Omega^c=E\times \mathbb{R}^{n-k}$, where $E$ is an
ellipsoid in
$\mathbb{R}^k$;

\item[(b)] If $\Omega^c$ is not contained in a strip, then   Lemma
\ref{TM4} and Theorem \ref{thm:1} imply that
$\Omega^c=\{x_1>f(x_{n-k+1},...,x_n)\}$, where $ f$ is a
convex and real analytic function.
\end{enumerate}

\bibliographystyle{amsplain}

\end{document}